\documentclass[10pt]{amsart}
\usepackage{a4}
\usepackage{amssymb}
\usepackage{amsmath}
\usepackage{amsfonts}
\usepackage{mathtools}
\usepackage{fancybox}
\usepackage{empheq}
\usepackage[stable]{footmisc}
\usepackage{calrsfs}
\usepackage{hyperref}
\usepackage{changes}
\usepackage{framed}
\usepackage{parskip}
\usepackage{color, colortbl}
\definecolor{LightCyan}{rgb}{0.88,1,1}
\definecolor{Gray}{gray}{0.9}
\usepackage{soul}

\usepackage{color}
\usepackage{comment}
\usepackage[all,cmtip]{xy}

\usepackage{cleveref}

\newtheorem{theorem}{Theorem}
\newtheorem{lemma}[theorem]{Lemma}

\newtheorem{proposition}[theorem]{Proposition}

\newtheorem{remark}[theorem]{Remark}

\newcommand{\Z}{\mathbb{Z}}

\newcommand{\Gal}{\mathrm{Gal}}

\title{A generalization of the Hasse-Arf theorem}

\date{\today}

\author[I. Tsouknidas]{Ioannis Tsouknidas}
\address{Beijing Institute of Mathematical Sciences and Applications (BIMSA), Huairou District, Beijing, 101408, China\\
}
\email{iotsouknidas@bimsa.cn}

\date \today

\makeatletter
\newcommand{\aprod}{\mathop{\operator@font \hbox{\Large$\ast$}}}
\makeatother

\begin{document}
\begin{abstract}
We use the theory of Harbater-Katz-Gabber curves to derive a generalization of the Hasse-Arf theorem for complete local field extensions in positive characteristic.
\end{abstract}

\maketitle
\section{Introduction}
Let $k$ be a field of positive characteristic $p\geq 5$ which is algebraically closed. Our aim is to establish a generalization of the Hasse-Arf theorem for wildly ramified extensions of the complete local field $k\left((x)\right)$. Namely we will prove an equivalence condition for the upper jumps of the ramification filtration to be integers. To that end we exploit the machinery of Harbater-Katz-Gabber curves, which
 are the complete analogues of local fields in number theory. A Harbater-Katz-Gabber curve can be seen as a tower of generalized Artin-Schreier extensions:
 \[
  F=F_h=F_{h-1}(\bar{f}_{h-1})>F_{h-1}>\dots>F_{2}=F_1(\bar{f}_1)>k((x)).
\]
For the elements $\bar{f}_1,\ldots,\bar{f}_{h-1}$ we have that $\bar{f}_i\in F_{i+1}$ and its minimal polynomial over $F_i$ is of the form:
\[
	P_i(X)-D_i
\]
for some $D_i\in F_i.$ In our generalization, we impose a condition on each $D_i$ for $i=2,\ldots,h-1$ which we explain after introducing the Hasse-Arf theorem.

For a complete discrete valuation ring $A_K$ with residue field $k$ and fraction field $K$ with equal characteristic $p$, there is a uniformizing parameter $x$ such that $A_K\simeq k\left[[x]\right]$ and $K=k\left((x)\right)$. For a finite Galois extension $F/K$, set $A_F$ for the integral closure of $A_K$ in $F$ and denote the maximal ideal of $A_F$ by $p_F.$ The Galois group $\Gal{(F/K)}$ contains the subgroups $G_i$ comprising $\sigma\in \Gal{(F/K)}$ which operate trivially on $A_F/p_F^{i+1}$. These subgroups establish the following group filtration:
\[
	\Gal{(F/K)}=G_{-1}\geq G_0\geq G_1\geq\ldots\geq {1}
\]
which stabilizes after finite terms. It is called the \emph{ramification filtration} of $\Gal{(F/K)}$. An integer $i$ such that $G_i>G_{i+1}$ is called a \emph{jump of the ramification filtration.} Denote the jumps which are greater than or equal to $1$ by $b_1,\ldots,b_{h-1}.$ Then:
\[
	\Gal{(F/K)}=G_{-1}\geq G_0\geq G_1=G_{b_1}>G_{b_2}>\ldots>G_{b_{h-1}}> {1}
\]
Now consider a subgroup $H$ of $\Gal{(F/K)}$. One wants to know the ramification filtration of $H$ and of $\Gal{(F/K)}/H$, if $H$ is normal. In the former case, it holds that $H_i=G_i\cap H$, therefore the ramification filtration of $\Gal{(F/K)}$ completely determines the ramification filtration of $H.$ In the latter case a similar result holds but one needs to modify the numbering:

Set $\phi:[-1,\infty)\to [-1,\infty)$ by
\[
	\phi(u)=\int_{0}^u\frac{dt}{[G_0:G_t]}\text{ for }u\geq 0
\]
and $\phi(u)=u$ for $u\in [-1,0].$

Define $G^{\phi(i)}:=G_i$. Then the ramification filtration of $\Gal{(F/K)}$ \emph{in upper numbering} is:
\[
	G^{-1}=G\geq G^0=G_0= G^{\phi(b_1)}>\ldots>G^{\phi(b_{h-1})}>1.
\]
The upper numbering allows for the calculation of the ramification groups of $\Gal{(F/K)}/H$ due to the formula:
\[
	\left(\frac{\Gal{(F/K)}}{H}\right)^i=\frac{G^iH}{H}.
\]
The Hasse-Arf theorem states that if $\Gal{(F/K)}$ is abelian then the jumps of its upper ramification filtration are integers. In other words, if $\Gal{(F/K)}$ is abelian then $\phi(b_i)\in \Z$ for all $i=1,\ldots,h-1.$

 Our main result is the following:

 \textbf{Thm. \ref{hassearfgeneralization}.} The jumps of the upper ramification filtration of $\Gal{(F/K)}$ are integers if and only if for $i=2,\ldots,h-1$ the following holds:
  $$\min D_i=f(x)\bar{f}_{i-1},$$
  where $f(x)\in k[x]$ and the minimum is taken with respect to the valuation at the extension in $F$ of the totally ramified point.

It is a work in progress to use our result in order to characterize the group $\Gal{(F/K)}$, which would be a semidirect product but not necessarily abelian.

\subsection*{Acknowledgment} I thank Aristides Kontogeorgis for helpful conversations and for reading the initial version of this paper. Also I thank Beihui Yuan for her help in computing explicit equations for Artin-Schreier-Witt extensions using Macauley 2.

\section{Background material}

A Harbater-Katz-Gabber curve (HKG for short) is a cover $\pi:X\to \mathbb{P}^1$ which has at most one tamely ramified point and one totally and wildly ramified point of $\mathbb{P}^1.$ Without loss of generality we assume that the wildly ramified point is $P_{\infty}$. Let $P\in X$ such that $\pi(P)=P_\infty$. This induces an extension of residue fields and under a suitable choice of uniformizer, the extension is $F/k((x))$. The valuation of $k((x))$ at $P_\infty$ is $v_{\infty}$, so that $v_\infty(x)=-1.$ The valuation $v_\infty$ extends to a valuation of $F$ which we denote by $v.$ Assume that $F/k((x))$ is finite and Galois. Pick a uniformizer $x'$ of $P$. Then the ramification groups of $\Gal{(F/k((x)))}$ are equivalently defined by:
\[
	G_i:=\{\sigma\in \Gal{(F/k((x)))}:v(\sigma x'-x')>i\}
\]
for $i\geq -1$. Since $v$ is discrete, the indices $i$ are integers. The groups $G_{-1}$, $G_0$ are called the decomposition and inertia ramification groups of $P$ over $P_\infty.$ Since the extension is Galois and $P_{\infty}$ ramifies totally, the decomposition group coincides with the elements of $\Gal{(F/k((x)))}$. Additionally since $k$ is algebraically closed, the decomposition group coincides with the inertia subgroup. Therefore $\Gal{(F/k((x)))}=G_{-1}=G_0$ and that group is isomorphic to the semidirect product $G_1\rtimes (G_0/G_1)$. The quotient $G_0/G_1$ is cyclic with prime-to-$p$ order and it corresponds to a Kummer extension of $k((x))$. The group $G_1$ is a $p$-group and the consecutive quotients $G_i/G_{i+1}$ for $i\geq 1$ are elementary abelian groups, that is, isomorphic to finite direct sums of $\Z/p\Z.$

By \cite[cor. 2, pg. 70]{SeL} we will assume that $\Gal{(F/k((x)))}=G_1$, thus discarding tame ramification. In that case the inertia subgroup of the ramification filtration coincides with the first ramification group. A good reference for properties of algebraic function field extensions can be found at \cite{Stichtenothv2009}.

Denote by $b_1,\ldots,b_{h-1}$ the jumps of the lower ramification filtration of $\Gal{(F/k\left((x)\right))}$, hence:
\[
	1\lneq G_{b_{h-1}}\lneq \ldots\lneq G_{b_1}=\Gal{(F/k\left((x)\right))}.
\]
Write $[G_{b_i}:G_{b_{i+1}}]=p^{n_i}$ for $i=1,\ldots, h-1$, where $G_{b_{h}}=1.$ The following results are due to the work of Karanikolopoulos and Kontogeorgis in \cite{Karanikolopoulos2013}. 

Setting $F_i=F^{G_{b_i}}$ yields the following extensions:
\[
  F=F_h>F_{h-1}>\dots>F_{2}>F_1=k((x)),
\]
and for each $i=2,\ldots,h-1$ there exists $\bar{f}_{i}\in F_{i+1}$ such that $F_{i+1}=F_i(\bar{f}_i).$  By \cite{1901.08446},
the field generators, $\bar{f}_i$, have minimal polynomials of the form
\begin{equation}
\label{defeqi}
X^{p^{n_i}}+a^{(i)}_{n_i-1}X^{p^{n_i-1}}+\dots+a^{(i)}_{0} X- D_i.
\end{equation}
where all the coefficients $a^{(i)}_{n_i-j}$, $j=1,\dots,n_1$ are in $k$ and $D_i\in F_i$.

 Attached to the extension is its Weierstrass semigroup, $H(P)$, which is the set of nonnegative integers $i$ such that $v(f)=-iP$ for some nonzero $f\in F.$ These integers are called \emph{pole numbers} of the extension $F/k(x)$ at $P.$ The valuation of each $\bar{f}_i$ is $v(\bar{f}_i)=-\bar{m}_i$ and by the construction of Karanikolopoulos and Kontogeorgis, the positive integers $\bar{m}_i$ satisfy the following conditions:
\begin{itemize}
\item $\bar{m}_i<\bar{m}_{i+1}$,
\item $\bar{m}_i=p^{n_{i+1}+\ldots+n_{h-1}}b_i=|G_{b_{i+1}}|b_i$,
\item $\bar{m}_i$ is a minimal generator of the Weierstrass semigroup $H(P)$,
\end{itemize}
for each $i=1,\ldots,h-1.$
 The Weierstrass semigroup is generated by the elements
$\{|G_0|, \bar{m}_1,\dots, \bar{m}_{h-1}\}$. The last subset of the Weierstrass semigroup might not be the minimal set of generators, since this depends on whether $G_1(P)$ equals $G_2(P)$, see \cite[thm. 13]{Karanikolopoulos2013}. The generators $\bar{m}_i$ give rise to a flag of vector spaces:
 \[
 k\subsetneq L(\bar{m}_1P)\subsetneq\ldots\subsetneq L(\bar{m}_{h-1}P)                           
 \]
 where $L(mP)=\{f\in F:(f)+mP\geq 0\}\cup {0}$ is the Riemann-Roch space associated to the divisor $mP.$ 

Consider a monomial
\[
	f:=x^{\ell_0}\bar{f}_1^{\ell_1}\ldots \bar{f}_{h-1}^{\ell_{h-1}}
\]
in $L(mP)$. Then we can assume that $\ell_i$, $i=1,\ldots,h-1$ is positive, otherwise $f$ would have a pole other than $P.$ Additionally we can assume that $\ell_i<p^{n_i}$, since we can use the minimal polynomial of $\bar{f}_i$ to substitute $\bar{f}_i^{p^{n_i}}$ accordingly.

If $g:=x^{\lambda_0}\bar{f}_1^{\lambda_1}\ldots \bar{f}_{h-1}^{\lambda_{h-1}}$ is another monomial in $L(mP)$, then \cite[lem. 2]{MR4194180} guarantees that their valuations differ. Therefore we can prove the following lemma:

\begin{lemma}\label{gradedring}
Assume that 
\[	
A
=
\sum_
{ \stackrel{\ell_0\in \mathbb{Z}}{ 0\leq\ell_j<p^{n_j}\text{ for }j=1,\ldots,i-1}}
 \gamma_{\ell_0,\dots,\ell_{i-1}}^{(i)}
 x^{\ell_0}
 \bar{f}_1^{\ell_1}
 \cdots
  \bar{f}_{i-1}^{\ell_{i-1}}
\]
is in $L(mP)$, for some $m>0$, where $\gamma_{\ell_0,\dots,\ell_{i-1}}^{(i)}\in k$. Then each monomial summand $ \gamma_{\ell_0,\dots,\ell_{i-1}}^{(i)}x^{\ell_0}\bar{f}_1^{\ell_1}\dots \bar{f}_{i-1}^{\ell_{i-1}}$ is in $L(mP).$
\end{lemma}
\begin{proof}
Every two monomial summands have different valuations, therefore there will be a minimal one, call it $\min A$, so by the strict triangle inequality we get $v(A)=v(\min A)$. On the other hand $v(A)\geq -m$ therefore for each monomial summand we have:
\[
	v(\gamma_{\ell_0,\dots,\ell_{i-1}}^{(i)}x^{\ell_0}\dots \bar{f}_{i-1}^{\ell_{i-1}})\geq v(\min A)\geq -m,
\]
as proclaimed.
\end{proof}

The following proposition appeared in \cite{1901.08446}:

 \begin{proposition}\label{module}
 For $m\in H(P)$ the following holds:
 \begin{equation}
\label{W-def}
L\big(mP \big)
=
\left\langle
x^{a_0}\bar{f}_1^{a_1}\cdots \bar{f}_s^{a_s}:
\begin{array}{l} 
0\leq a_0,\,0 \leq a_i < p^{n_i} \text{ for } 1\leq i \leq s,
\\
\text{ and }
v(x^{a_0}\bar{f}_1^{a_1}\cdots \bar{f}_s^{a_s})
\geq
 -m
\end{array}
\right\rangle_k.
\end{equation} 
The integer $s$ is the greatest index 
of $\bar{m}_i$ such that $\bar{m}_i\leq m$.
 \end{proposition}

\begin{remark}

The condition $a_0 \geq 0$ does not appear in the original result in the cited publication. It was an unintended omission. One can prove that if an element in $L(mP)$ has a negative exponent in $x$ then will have a pole different than $P$. In fact even less is needed: By definition, the vector space $L(mP)$ (assume without loss of generality that $m>m_{h-1}$ ) is associated with the elements of the Weierstrass semigroup $H(P)$ by taking valuations at $P$. Since the Weierstrass semigroup $H(P)$ is generated by the elements $|\Gal{(F,k\left((x)\right))}|,\bar{m}_1,\ldots,\bar{m_{h-1}}$ under addition, one gets that the elements of $L(mP)$ is generated by the elements $x,\bar{f}_1,\ldots,\bar{f}_{h-1}$, and their products without involving inverses.
\end{remark}

 Eq. (\ref{defeqi}) vanishes at $\bar{f}_i$, yielding the equality:
\begin{equation}\label{eqoffi}
\bar{f}_i^{p^{n_i}}+a^{(i)}_{n_i-1}\bar{f}_i^{p^{n_i-1}}+\dots+a^{(i)}_{0} \bar{f}_i= D_i.
\end{equation}
The information of the successive extensions is encoded in the coefficients $a^{(i)}_{j}$ of the additive left part and in the constant terms $D_i\in F_{i}$.

Taking valuations on both sides gives  
  that the pole divisor of
$D_{i}$ is $p^{n_{i}}\bar{m}_{i}P$.

Therefore we have two descriptions for $D_i$. On the one hand, $D_i\in L(p^{n_{i}}\bar{m}_{i}P)$. On the other, $D_i\in F_i=k\left((x)\right)(\bar{f}_1,\dots,\bar{f}_{i-1}   )$. We use both descriptions to express $D_i$ as:
\begin{equation}
\label{Ddef}
D_{i}
=
\sum_{ \stackrel{\ell_0\in \mathbb{N}}{ 0\leq\ell_j<p^{n_j}\text{ for }j=1,\ldots,i-1}}
 \gamma_{\ell_0,\dots,\ell_{i-1}}^{(i)}x^{\ell_0}\bar{f}_1^{\ell_1}\cdots \bar{f}_{i-1}^{\ell_{i-1}}.
\end{equation}

 Namely using $D_{i}\in F_i$ allows one to write: 
 \[
 D_i=\sum f(x)\bar{f}_1^{\ell_1}\cdots \bar{f}_{i-1}^{\ell_{i-1}},
 \]
 where $f(x)\in k\left((x)\right)$. Expanding $f$ as a formal sum while recalling that $D_i$ has a single pole at $P$, combined with lemma \ref{gradedring} and proposition \ref{module} yields the desired expression.

We recall the following lemma from \cite[lem. 3]{MR4194180} which also appears as lemma 3 in \cite{Madden78}.

\begin{lemma}\label{mainlem}
Let $F=F_{h}$ as before with generators $x,$ $\bar{f}_i$, $i=1,\dots,h-1$ and associated minimal polynomials $A_i$ as in equation (\ref{defeqi}):
\[
A_i(X)=X^{p^{n_i}}+a_{n_i-1}^{(k)} X^{p^{n_i-1}}+\dots+a_0^{(k)} X- D_i,
\]
where  $D_i$ is given  in equation (\ref{Ddef}),

\[
D_{i}
=
\sum_{ \stackrel{\ell_0\in \mathbb{N}}{ 0\leq\ell_j<p^{n_j}\text{ for }j=1,\ldots,i-1}}
 \gamma_{\ell_0,\dots,\ell_{i-1}}^{(i)}x^{\ell_0}\bar{f}_1^{\ell_1}\cdots \bar{f}_{i-1}^{\ell_{i-1}}.
\]

 Then  one of the monomials $x^{\ell_0}\bar{f}_1^{\ell_1}\cdots \bar{f}_{i-1}^{\ell_{i-1}}$ has also pole divisor
$p^{n_i}\bar{m}_i P$  and this holds for all $i=1,\dots,h-1$.
\end{lemma}

For an element $\sigma\in \Gal{(F/k\left((x)\right))}$ we denote the action on a field generator $\bar{f_i}$ by
\begin{equation}\label{actionofgroup}
\sigma(\bar{f}_i)  =\bar{f}_i+\bar{C_i}(\sigma),
\end{equation}
where $\bar{C_i}(\sigma)$ has valuation smaller than $\bar{m}_i$, therefore $\bar{C_i}(\sigma)\in L\left(\bar{m}_{i-1}P\right)$.
 Since only one place is allowed to ramify, it is known (see \cite[prop. 27]{Karanikolopoulos2013}) that if $\sigma\mid_{F_i}=\mathrm{id}$ and $\sigma\mid_{F_{i+1}}\neq \mathrm{id}$ then
\[
  \sigma(\bar{f}_j)=\bar{f}_ j\text{ for }j<i\text{ and }\sigma(\bar{f}_i)=\bar{f}_i+c,\, c\in k^*.
\]


\section{The generalization of the Hasse-Arf theorem}
In this section we give a generalization of the Hasse-Arf theorem. Namely we prove that for an HKG curve, the jumps of the upper ramification filtration are integers if and only if a certain condition holds for the constant term of the minimal polynomial of each field generator.

Consider an HKG curve given as before in the form of consecutive extensions:
\[
	F=F_h=F_{h-1}(\bar{f}_{h-1})>F_{h-1}>\dots>F_2=F_1(\bar{f}_1)>k((x)).
\]
The field generators satisfy equations given in (\ref{eqoffi}).
The pole divisor of each $\bar{f}_i$ is $m_iP$ and $m_i=p^{n_{i+1}+\dots+n_{h-1}}b_i$ and $b_i$ are the jumps of the lower ramification filtration.

For the minimal polynomial of $\bar{f}_i$ we consider the constant term $D_i$. Denote its minimal term as:
\[
\min D_i=a_ix^{\nu_{i,0}}\dots \bar{f}_{i-1}^{\nu_{i,i-1}},\]
$a_i$ is in $k^*$. As explained before it holds that $\nu_{i,j}\geq 0$, $\nu_{i,j}<p^{n_j}$ for $j>0$ and all $i.$ The following is the main theorem of this paper:

\begin{theorem}\label{hassearfgeneralization}
	The jumps in the upper ramification filtration of $G$ are integers if and only if the following condition holds:
	$$\min{D_i}=a_ix^{\nu_{i,0}}\bar{f}_{i-1}$$
	for $i>1.$
\end{theorem}
In other words, for $i=2,\dots,h-1$ it holds that $\nu_{i,j}=0$ for $j=1,\dots,i-2$ and $\nu_{i,i-1}=1$.

\begin{proof}Before proving the theorem some preparations are in order. The constant term $D_i$ has valuation equal to the valuation of $\bar{f}_i^{p^{n_i}}$. On the other hand, $v(D_i)=v(\min D_i)=v(ax^{\nu_{i,0}}\dots \bar{f}_{i-1}^{\nu_{i,i-1}})$ for $i>1.$

The two equalities give:
\[
	p^{n_i}\bar{m}_i=\nu_{i,0}|G|+\nu_{i,1}\bar{m}_1+\ldots+\nu_{i,i-1}\bar{m}_{i-1}.
\]
Since $\bar{m}_i=p^{n_{i+1}+\ldots+n_{h-1}}b_i$ for each $i$ and $|G|=p^{n_1+\ldots+n_{h-1}}$ we have:
\[
	p^{n_i}p^{n_{i+1}+\ldots+n_{h-1}}b_i
	=
	\nu_{i,0}p^{n_1+\ldots+n_{h-1}}+\nu_{i,1}p^{n_2+\ldots+n_{h-1}}b_1+\ldots+\nu_{i,i-1}p^{n_{i}+\ldots+n_{h-1}}b_{i-1}
\]
Cancel the common part, i.e. $p^{n_i+\ldots+n_{h-1}}$, from both equations to get:
\[
	b_i=\nu_{i,0}p^{n_1+\ldots+n_{i-1}}+\nu_{i,1}p^{n_2+\ldots+n_{i-1}}b_1+\ldots+\nu_{i,i-1}b_{i-1}.
\]
Substract $b_{i-1}$ from both sides and we have:
\begin{equation}\label{eqofbi}
b_i-b_{i-1}=\nu_{i,0}p^{n_1+\ldots+n_{i-1}}+\nu_{i,1}p^{n_2+\ldots+n_{i-1}}b_1+\ldots+(\nu_{i,i-1}-1)b_{i-1}.
\end{equation}
This holds for $i=2,\ldots,h-1$. Also remember that all $\nu_{i,j}$ are nonnegative and also $\nu_{i,j}<p^{n_j}$ for $j=1,\ldots,i-1$ and all $i.$

Suppose first that the upper jumps are integers. This means that $\phi(b_i)\in \Z$ for each $i$, where:
\[
	\phi(u)=\int_{0}^u\frac{dt}{[G_0:G_t]}.
\]
We will prove the argument directly but we exhibit first the process for the first two terms for convenience. For the first term $D_1$ we don't have anything to show. For $D_2$ we have that $\phi(b_2)=b_1+\frac{b_2-b_1}{p^{n_1}}.$

Since $\phi(b_2)$ is integer we get that $p^{n_1}$ divides $b_2-b_1.$ Therefore $p^{n_1}$ will divide the right hand side of equation \ref{eqofbi} for $i=2$ giving:
\[
	p^{n_1}\mid \nu_{2,0}p^{n_1}+(\nu_{2,1}-1)b_{1}.
\]
Since $0\leq \nu_{2,1}<p^{n_1}$ and $(b_1,p)=1$, we get that $\nu_{2,1}$ must be equal to $1$. Therefore $\min D_2=a_2x^{\nu_{2,0}}\bar{f}_2.$

For the general case we have that $\phi(b_i)\in \Z$ for $i=1,\ldots,h-1$ giving that 
\begin{align*}\label{hassearfdivisibility}
p^{n_1}&\mid b_2-b_1\\
&\vdots\\
p^{n_1+\ldots+n_i}&\mid b_{i+1}-b_i\tag{*}\\
&\vdots\\
 p^{n_1+\ldots+n_{h-2}}&\mid b_{h-1}-b_{h-2}.
\end{align*} 

We will show that $\mid D_i=a_ix^{\nu_{i,0}}\bar{f}_{i-1}$. Since $p^{n_1+\ldots+n_{i-1}}$ divides $b_i-b_{i-1}$, it divides the right hand side of eq. \ref{eqofbi}, which is: 
$$\nu_{i,0}p^{n_1+\ldots+n_{i-1}}+\nu_{i,1}p^{n_2+\ldots+n_{i-1}}b_1+\ldots+(\nu_{i,i-1}-1)b_{i-1}.$$
 The first term is already a multiple of $p^{n_1+\ldots+n_{i-1}}$ therefore we arrive at:
\begin{equation}\label{eqdivisibilityofbi}
p^{n_1+\ldots+n_{i-1}}\mid \nu_{i,1}p^{n_2+\ldots+n_{i-1}}b_1+
	\ldots+\nu_{i,i-2}p^{n_{i-1}}b_{i-2}+(\nu_{i,i-1}-1)b_{i-1}.
\end{equation}
Recall that $0\leq \nu_{i,j}<p^{n_j}$ for $j>0$. First we show that $\nu_{i,i-1}$ is $1.$ For this write: 
\[
	\tau p^{n_1+\ldots+n_{i-1}}= \nu_{i,1}p^{n_2+\ldots+n_{i-1}}b_1+\ldots+\nu_{i,i-2}p^{n_{i-1}}b_{i-2}+(\nu_{i,i-1}-1)b_{i-1}.
\]
to get to 
\[
	(\nu_{i,i-1}-1)b_{i-1}=\tau p^{n_1+\ldots+n_{i-1}}-( \nu_{i,1}p^{n_2+\ldots+n_{i-1}}b_1+\ldots+\nu_{i,i-2}p^{n_{i-1}}b_{i-2}).
\]
The right hand side is a multiple of $p^{n_{i-1}}$ whereas $\nu_{i,i-1}$ is less than $p^{n_{i-1}}$ and $b_{i-1}$ is prime to $p$. From this, $\nu_{i,i-1}=1.$  

Going back, we rewrite eq. \ref{eqdivisibilityofbi} and canceling $p^{n_{i-1}}$ from both sides to get that:
\[
	p^{n_1+\ldots+n_{i-2}}\mid \nu_{i,1}p^{n_2+\ldots+n_{i-2}}b_1+
	\ldots+\nu_{i,i-3}p^{n_{i-2}}b_{i-3}+\nu_{i,i-2}b_{i-2}.
\]
Then $p^{n_{i-2}}$ divides $\nu_{i,i-2}b_{i-2}$ and the same argument now gives that $\nu_{i,i-2}=0.$

Repeating the process will give that $p^{n_1+\ldots+n_{i-2}}\mid \nu_{i,1}p^{n_2+\ldots+n_{i-2}}b_1$ and since $(b_1,p)=1$ and $\nu_{i,1}<p^{n_1}$ we get that $\nu_{i,1}=0$. Therefore $\nu_{i,1}=\ldots=\nu_{i,i-2}=0$ and $\nu_{i,i-1}=1$ yielding that
\[
	\min D_i=a_ix^{\nu_{i,0}}\bar{f}_{i-1},
\]
for all $i=2,\ldots,h-1.$

Now we proceed to show that if $\min D_i=a_ix^{\nu_{i,0}}\bar{f}_{i-1}$ for $i=2,\ldots,h-1$ then the jumps of the upper ramification filtration of $G$ are integers. For the jumps $b_i$ of the lower ramification filtration, eq. \ref{eqofbi} becomes:
\begin{equation}\label{neweqofbi}
b_i-b_{i-1}=\nu_{i,0}p^{n_1+\ldots+n_{i-1}}.
\end{equation}
for $i=2,\ldots,h-1.$

Then 
\[
	\phi(b_i)=b_1+\frac{b_2-b_1}{p^{n_1}}+\ldots+\frac{b_i-b_{i-1}}{p^{n_1+\ldots+n_{i-1}}}
\]
and the result is imminent by (\ref{neweqofbi}).
\end{proof}

 \bibliographystyle{plain}
\bibliography{C:/Users/kharg/Desktop/tsoukupdwindows.bib}

\end{document}